\documentclass[11pt,a4paper]{article}
\usepackage[top=2.5cm,bottom=2.5cm,left=2.2cm,right=2.2cm]{geometry}
\usepackage[T1]{fontenc}
\usepackage[utf8]{inputenc}
\usepackage{color}
\usepackage{graphicx}
\usepackage{amsfonts}
\usepackage{extarrows}
\usepackage{amsmath,amsthm,amssymb,color}
\usepackage{hyperref}
\usepackage{eepic}
\usepackage{lineno}
\usepackage{enumerate}	
\usepackage[utf8]{inputenc}
\usepackage{paralist}
\usepackage{cite}
\usepackage{algorithm}
\usepackage{algorithmicx}
\usepackage{algpseudocode}
\usepackage{authblk}
\usepackage{mathtools}
\usepackage{pifont}
\usepackage[numbers,sort&compress]{natbib}
\renewcommand{\Authand}{~~~~~~~~~~~~~~~~~~~}
\usepackage{comment}
\usepackage{ulem}

\usepackage{tikz}
\usetikzlibrary{positioning}

\hypersetup
{
	colorlinks=true,
	linkcolor=blue,
	filecolor=blue,
	urlcolor=blue,
	citecolor=cyan,
}

\def\blue{\color{blue}}
\def\red{\color{red}}

\usepackage{etoolbox}

\newcommand{\iso}{\simeq}
\newcommand{\niso}{\not \simeq}
\newcommand{\sm}{\setminus}
\newcommand{\cl}{{\rm cl}}
\newcommand{\clm}{{\rm cl}^M}
\newcommand{\claw}{K_{1,3}}
\newcommand{\co}{{\rm co}}
\newcommand{\Lp}{L^{-1}}
\newcommand{\mH}{\mathcal{H}}

\newtheorem{theorem}{Theorem}[section]

\newtheorem{conjecture}[theorem]{Conjecture}

\newtheorem{property}[theorem]{Property}

\def\b{\color{blue}}
\def\r{\color{red}}

\theoremstyle{definition}

\newtheorem{claim}{Claim}

\renewcommand{\thecase}{\arabic{case}}

\AtBeginEnvironment{proof}{\setcounter{case}{0}}
{\setlength{\leftmargini}{1.5\parindent}
	\begin{itemize}
		\setlength{\itemsep}{-1.1mm}}
	{\end{itemize}}

\renewcommand{\baselinestretch}{1.20}

\renewcommand\Authsep{\quad}     
\renewcommand\Authands{\quad}    

\newcommand{\imgi}{{\bf i}}
\newcommand{\e}{{\bf e}}
\begin{document}
	\title{\bf A Note on  Gr\"{u}nbaum's Conjecture about  Longest Cycles and Paths}
	\author[1]{\bf Masaki Kashima \footnote{Email: masaki.kashima10@gmail.com.}}
	\author[2]{\bf Kenta Ozeki \footnote{Email: ozeki-kenta-xr@ynu.ac.jp.}}
	\author[3]{\bf Leilei Zhang \footnote{Email: mathdzhang@163.com.}}
	
	\affil[1]{\footnotesize Faculty of Science and Technology, Keio University, Yokohama 223-8522, Japan}
	\affil[2]{\footnotesize Faculty of Environment and Information Sciences, Yokohama National University, Yokohama 240-8501, Japan}
	\affil[3]{\footnotesize School of Mathematics and Statistics, Key Laboratory of Nonlinear Analysis \& Applications (Ministry of Education), Central China Normal University, Wuhan 430079, China}
	\date{}
	\maketitle
	\begin{abstract}
		
Let $c(G)$ denote the circumference of a graph $G$, i.e., the number of vertices in its longest cycle.  For positive integers $n$ and $k$ with $n>k$, let $\varGamma(n;k)$ be the class of graphs of order $n$ with $c(G) = n - k$ such that every induced subgraph of order $n - k$ is Hamiltonian. When $k = 1$, the class $\varGamma(n; 1)$ coincides with the family of hypohamiltonian graphs—non-Hamiltonian graphs in which the deletion of any single vertex yields a Hamiltonian graph. Replacing Hamiltonian with traceable and $c(G)$ with $p(G)$, the order of a longest path, defines the analogous class $\varPi(n;k)$.  Gr\"{u}nbaum (1974) conjectured that both $\varGamma(n; k)$ and $\varPi(n; k)$ are empty for all $n > k \ge 2$. In this note, we first establish upper bounds on the maximum degree of graphs in the classes $\varGamma(n; k)$ and $\varPi(n; k)$. Using these bounds, we show that $\varGamma(n; k)$ is empty when $n < k^2 + 2k + 3$, and that $\varPi(n; k)$ is empty when $n < k^2 + 2k + 2$. These results provide further evidence supporting Gr\"{u}nbaum’s conjecture.
		
		
	\end{abstract}

	\section{Introduction}
	
	We adopt standard notation and terminology in graph theory; undefined terms follow Bondy and Murty \cite{BM}. Throughout, all graphs are assumed to be simple, finite, and undirected. The {\it circumference}  $c(G)$ of a graph $G$ is the length of a longest cycle in $G$.
	A {\it Hamilton cycle} (resp. {\it Hamilton path }) in a graph is a cycle (resp. path) that passes through all vertices. A graph containing a Hamilton cycle is said to be {\it Hamiltonian}. 
	
	In 1974, Gr\"unbaum \cite{G} introduced the family $\Gamma(k,k)$, consisting of all graphs whose circumference equals their order minus $k$, and in which every set of $k$ vertices is avoided by some longest cycle.  In particular, $\Gamma(1,1)$ coincides with the class of {\it hypohamiltonian graphs}—non-Hamiltonian graphs in which the deletion of any single vertex yields a Hamiltonian graph. These graphs have been extensively studied due to their close connection with Hamiltonicity \cite{W,Z,JMOPZ}. In the same paper, Gr\"unbaum \cite{G} formulated the following conjecture concerning the class  $\Gamma(k,k)$.
	
	\begin{conjecture}{\rm (Gr\"{u}nbaum, 1974)}\label{conj1}
		For any integer $k\geq 2$, $\Gamma(k,k)$ is empty.
	\end{conjecture}
	
	The conjecture asserts that any graph  of order $n$ with circumference $n-k$ for $k \ge 2$ has an induced subgraph of order $n-k$ that is non-Hamiltonian. Despite its apparent simplicity, this conjecture has remained open for decades. As Zamfirescu pointed out in \cite{Z1}, “{\it Very little is known about the veracity of this conjecture, which in its general form seems disconcertingly difficult.}” Thomassen \cite{T} investigated the case $k = 2$, and supported Gr\"{u}nbaum's intuition by conjecturing that    $\Gamma(2,2)=\emptyset$. He further observed that any graph in   $\Gamma(2,2)$ must have all of its vertex-deleted subgraphs lying in  $\Gamma(1,1)$, implying that every such subgraph is hypohamiltonian. In a related direction,  Zamfirescu \cite{Z1} and  Goedgebeur, Renders, Wiener and Zamfirescu \cite{GRWZ} studied the broader class of $K_2$-Hamiltonian graphs, namely, graphs in which the removal of any pair of adjacent vertices results in a Hamiltonian graph. They constructed various examples of such graphs and investigated their properties. More recently, Zamfirescu \cite{Z2} showed that for any positive integers $c$ and $k$, there exists an infinite family of $c$-connected graphs with circumference exactly $n - k$, such that, as the order of the graphs tends to infinity, the ratio between the number of $k$-vertex subsets whose deletion results in a Hamiltonian graph and the total number of $k$-vertex subsets tends to $1$. For further results related to these conjectures, we refer the reader to \cite{T1,Z3,KKPS}. 
	
	
	For positive integers $n$ and $k$ with $n\geq k$, let $\varGamma(n;k)$ be the set of all graphs of order $n$ such that $c(G)=n-k$ and every induced subgraph with $n-k$ vertices is Hamiltonian. By definitions, it follows that $\Gamma(k,k)=\bigcup_{n=k}^{\infty}\varGamma(n;k)$ for any $k\geq 1$. For hypohamiltonian graphs, Holton and Sheehan \cite[p.223]{HS} proved that the maximum degree of such a graph cannot exceed $(n-4)/2$. This classical result provides a useful tool for the analysis of hypohamiltonian graphs. In this paper, we first establish an upper bound on the maximum degree of graphs in $\varGamma(n; k)$. 
	
	\begin{theorem}\label{thm3}
		Let $n$ and $k$ be positive integers. If $G \in \varGamma(n; k)$, then
		$$
		\Delta(G) \leq \frac{n - k^2 +1}{2} .
		$$   
	\end{theorem}
	
	The question “{\it For which values of n do there exist hypohamiltonian graphs on n vertices ?} ” is a fundamental and interesting problem in the study of hypohamiltonian graphs, and has therefore attracted considerable attention from many researchers. In 1964, Gaudin, Herz and Rossi \cite{GHR} proved that there exists no hypohamiltonian graph of order less than 10, and that the Petersen graph is the only hypohamiltonian graph with order 10. This can be regarded as one of the earliest results in the study of hypohamiltonian graphs. Subsequently, researchers discovered hypohamiltonian graphs of orders $n = 13, 15, 16$ and $n \ge 18$, and proved that no hypohamiltonian graphs exist for $n = 11, 12$ or $14$. For more details, see (\cite[p223-229]{HS}). Finally, in 1997, Aldred, McKay and Wormald \cite{AMW} proved that no hypohamiltonian graph exists on 17 vertices, thus settling the problem completely. Thomassen made numerous important contributions to the study the order of hypohamiltonian graphs \cite{T2,T1}.  Further results and related discussions on this question can be found in \cite{JMOPZ,T3}.
	
	Motivated by Conjecture~\ref{conj1} and the result of Gaudin, Herz and Rossi \cite{GHR}, in this note, we investigate a lower bound on the order of graphs in $\varGamma(n; k)$ for $k \ge 2$, if such graphs exist. In fact, since every graph in $\varGamma(n; k)$ is $(k + 2)$-connected, its minimum degree must be at least $k + 2$ (see Property \ref{lem2} in this paper). By applying Theorem \ref{thm3}, we can easily obtain the following theorem.
	
	\begin{theorem}\label{thm1}
		For any positive integers $n$ and $k$, if $n< k^2+2k+3$, then $\varGamma(n; k) = \emptyset.$
	\end{theorem}
	
	Following Kapoor, Kronk and Lick \cite{KKL}, we call a longest path in a graph $G$ a {\it detour} of $G$. For positive integers $n$ and $k$, let $\varPi(n; k)$ be the set of graphs $G$ of order $n$ such that the detour order   of $G$ is equal to $n-k$ and $G-S$ has a Hamilton path for any subset $S \subseteq V (G)$ of order $k$. Note that the class $\varPi(n; 1)$ consists of hypotraceable graphs, i.e., non-traceable graphs of which every vertex-deleted subgraph is traceable. In this paper, we also establish an upper bound for the maximum degree of graphs belonging to $\varPi(n; k)$.
	\begin{theorem}\label{thm4}
		Let $n$ and $k$ be positive integers. If $G \in \varPi(n; k)$, then
		$$
		\Delta(G) \leq \frac{n - k^2 }{2} .
		$$   
	\end{theorem}
	
	Gr\"{u}nbaum \cite{G} also conjectured that $\varPi(n;k)=\emptyset$ for all $n$ and $k$ with $n\geq k\geq 2$. In this paper, by using Theorem \ref{thm4}, we confirm this conjecture for the case $n<k^2+2k+2$.
	
	\begin{theorem}\label{thm2}
		For any positive integers $n$ and $k$, if $n<k^2+2k+2$, then $\varPi(n; k) = \emptyset.$
	\end{theorem}
	
	\section{Proofs of the Main Results}
	
	In this section, we present  proofs of our main results. The neighborhood of a vertex $x$ in a graph $G$ is denoted by $N(x)$ or $N_G (x)$, and the closed neighborhood of $x$ is $N[x]=N(x) \cup\{ x \}$. The degree of $x$ is denoted by $d(x)$. We denote by $\delta(G)$ and $\Delta(G)$ the minimum degree and maximum degree of $G$, respectively. For a vertex subset $ S \subseteq V (G)$, we use $G[S]$ to denote the subgraph of $G$ induced by $S$, and use $N(S)$ to denote the neighborhood of $S$; i.e., $ N(S) = \{\, y \in V(G) \setminus S \mid y \text{ has a neighbor in } S \,\} $. Given two vertex-disjoint subgraphs $S$ and $T$ of $G$, we denote by  $E_G(S,T)$ the set of edges having one endpoint in $S$ and the other in $T$. 
	
	We use $P[u, v] $ to denote the subpath of a path $P$ with endpoints $u$ and $v$. For a directed cycle $\overrightarrow{C}$, we use $\overrightarrow{C}[u, v]$ to denote the subpath of $\overrightarrow{C} $ from $u$ to $v$ along the direction of $\overrightarrow{C}$. 
	We begin with some structural properties that provide a necessary condition for any graph in the class $\varGamma(n; k)$. 
	
	\begin{property}\label{lem1}
		Let $n$ and $k$ be a positive integer with $n\geq k\geq 2$. If $G\in \varGamma(n;k)$, then for every vertex $u_1$ of $G$, $G$ has an induced path of order at least $k+1$ with an endpoint $u_1$.
	\end{property}
	
	\noindent{\it Proof.}  Let $u_1$ be a vertex in $G.$ Assume to the contrary that $G$ contains no induced path of order at least $k + 1$ with an endpoint $u_1$. Let $P= u_1 u_2 \cdots u_\ell$ be a longest induced path in $G$. Since $\ell \leq k$, we may select a set $S \subseteq V(G)$ of cardinality $k$ such that $u_\ell \notin S$ and $V(P) \setminus \{u_\ell\} \subseteq S$. Let $C$ be a Hamilton cycle in $G - S$, and let $v$ be a neighbor of $u_\ell$ in $C$. Since $v \notin V(P)$, by the maximality of $P$, there must exist a vertex $u_i \in V(P) \setminus \{u_\ell\}$ such that $vu_i \in E(G)$. Then the cycle obtained by replacing the edge $u_\ell v$ in $C$ with the path $P[u_\ell, u_i] \cup u_i v$ has length at least $n - k + 1$, contradicting the assumption that the circumference of $G$ is exactly $n - k$. \hfill $\Box$
	\bigskip
	
	Graphs in $\varGamma(n; k)$ also exhibit high connectivity. Indeed, note that if any set of $k$ vertices is removed from $G \in \varGamma(n; k)$, the remaining graph is Hamiltonian, and hence 2-connected. This implies that the connectivity of $G$ is at least $k+2$. We formalize this observation in the following property.
	
	\begin{property}\label{lem2}
		Let $G$ be a graph in $\varGamma(n; k)$. Then $G$ is $(k+2)$-connected.
	\end{property}
	
	\medskip
	
	With Properties~\ref{lem1} and~\ref{lem2}, we are now ready to prove our main theorem. \\
	
	\begin{figure}[!ht]
		\centering
		\includegraphics[width=115mm]{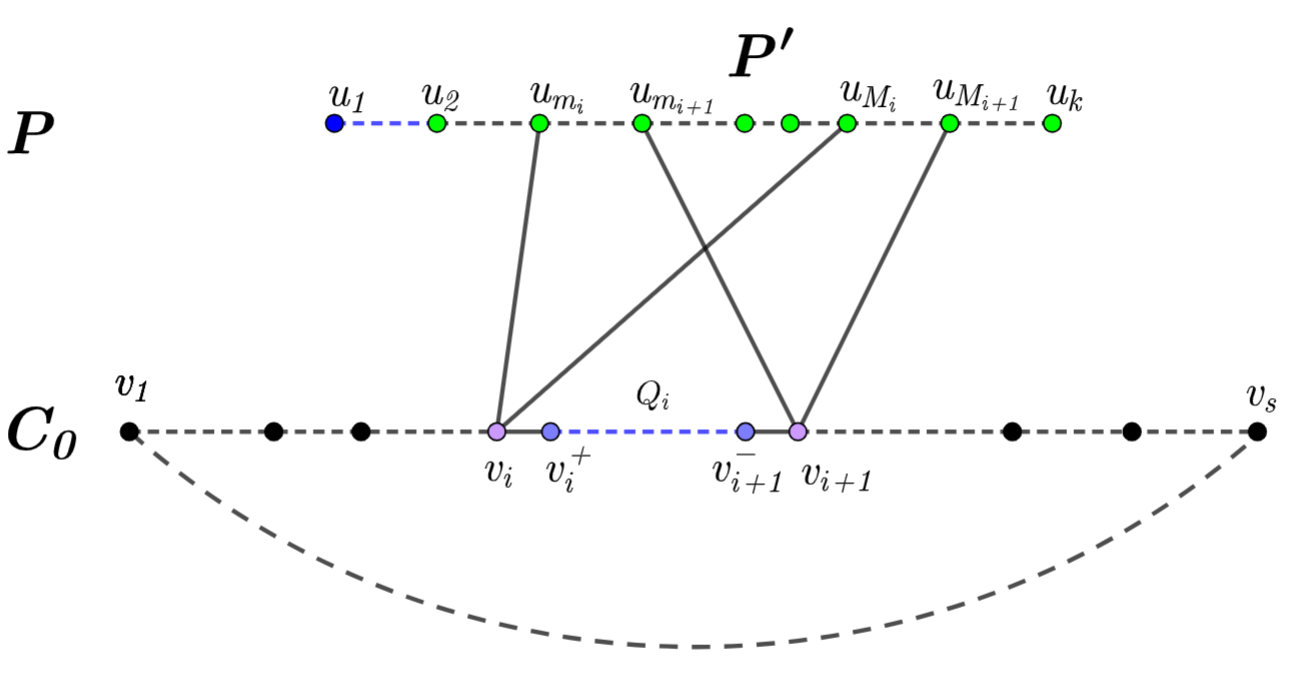}
		\caption{The induced path  $P,$ $P'$ and the cycle  $C_0$ in the proof of Theorem \ref{thm3}.}
		\label{fig1}
	\end{figure}
	
	\noindent{\it Proof of Theorem \ref{thm3}.}  Let $G \in \varGamma(n; k)$. By the definition of $\varGamma(n; k)$, the circumference of $G$ is $n - k$. By Property~\ref{lem2}, $G$ is $(k + 2)$-connected, and hence its minimum degree is at least $k+2$.
	
	Let $u_1$ be a vertex of $G$ with maximum degree. By Property \ref{lem1}, the graph $G$ contains an induced path of order at least $k$ with an endpoint $u_1$. Let $P$ denote such an induced path, and label its vertices consecutively as $u_1, u_2, \ldots, u_{k}$, where $u_i \sim u_{i+1}$ for all $1 \leq i \leq k-1$.  As $G\in \varGamma(n; k)$, $G-V(P)$ has a Hamilton cycle $C_0$, see Fig.\ref{fig1}.  Fix an orientation of $C_0$, and let $\overrightarrow{C_0}$ denote the cycle $C_0$ with this orientation. For each vertex $v \in V(C_0)$, we denote by $v^-$ the predecessor of $v$ and by $v^+$ the successor of $v$ with respect to the orientation of the cycle $\overrightarrow{C_0}$.   Set $P'=P-u_1$. Let $N_G(P') \cap V(C_0) = \{v_1, v_2, \ldots, v_s\},$ where the indices are ordered increasingly along the orientation $\overrightarrow{C_0}$.
	Since $G$ is $(k+2)$-connected, we have $s \geq k+1$. 
	For each $v_i \in N_G(P') \cap V(C_0)$, define
	$$N_{P}(v_i) = N_G(v_i) \cap V(P), \quad d_{P'}(v_i) = |N_{P}(v_i) \setminus \{u_1\}|,
	\quad \text{and}  \quad \varepsilon_i = 
	\begin{cases} 
		1 & \text{if $u_1$ is a neighbor of $v_i$,} \\
		0 & \text{otherwise}.
	\end{cases} 
	$$
	Note that 
	$|N_{P}(v_i)| = d_{P'}(v_i) + \varepsilon_i$,
	and $\sum_{i=1}^s \varepsilon_i \leq s$.
	For each $i \in \{1,2,\ldots,s\}$, let $Q_i = \overrightarrow{C_0}[v_i^+,v_{i+1}^-]$, 
	where $v_{s+1} = v_1$,
	and let $r_i = |V(Q_{i}) \cap N_G(u_1)|$.
	Thus, we have
	$\Delta(G) = d_G(u_1) = 1+ \sum_{i=1}^s (r_i + \varepsilon_i)$.

	The following claim is crucial for our proof.
	
	\begin{claim}\label{claim:Qi}
		For each $i \in \{1, \ldots, s\}$, we have
		\begin{equation*}\label{eq1}
			|V(Q_i)| \geq \frac{d_{P'}(v_i) + d_{P'}(v_{i+1})+ \varepsilon_i + \varepsilon_{i+1}}{2} + 2r_i,
		\end{equation*}
		where $\varepsilon_{s+1} = \varepsilon_1$.
	\end{claim}
	\begin{proof}
		We prove this claim considering two cases,
		according to whether $r_i = 0$ or $r_i \geq 1$.
		
		Suppose first $r_i =0$,
		that is, $u_1$ has no neighbor in $Q_i$.
		Let $m_i$ and $M_i$ denote the smallest and largest indices, respectively, such that $u_{m_i}, u_{M_i} \in N_{P}(v_i)$. Note that $M_i \geq 2$, and $m_i = 1$ if and only if $\varepsilon_i =1$.
		Since all vertices in $N_{P}(v_i)$ are contained in $P[u_{m_i}, u_{M_i}]$,
		we have 
		$$M_i - m_i + 1 \geq |N_{P}(v_i)| = d_{P'}(v_i) + \varepsilon_i.$$ 
		Similarly, let $m_{i+1}$ and $M_{i+1}$ denote the smallest and largest indices such that $u_{m_{i+1}}, u_{M_{i+1}} \in N_{P}(v_{i+1})$,
		which implies $M_{i+1} - m_{i+1}+ 1 \geq d_{P'}(v_{i+1}) + \varepsilon_{i+1}$.
		Thus, we deduce that
		\begin{eqnarray*}\label{eq3}
			\text{either} \quad M_{i+1} - m_i + 1 &\geq& \frac{1}{2}\left(d_{P'}(v_i)+ d_{P'}(v_{i+1})+ \varepsilon_i + \varepsilon_{i+1}\right),
			\\
			\text{or} \quad M_i- m_{i+1} + 1 &\geq& \frac{1}{2} \left(d_{P'}(v_i) + d_{P'}(v_{i+1})+ \varepsilon_i +\varepsilon_{i+1}\right).
		\end{eqnarray*}\label{eq3}
		Without loss of generality, we may assume that
		the former occurs.
		Since $G$ has circumference $n-k,$ the following cycle has order at most $n-k$.
		$$
		v_i u_{m_i} \cup P[u_{m_i},u_{M_{i+1}} ]\cup u_{M_{i+1}} v_{i+1}\cup \overrightarrow{C_0}[ v_{i+1}, v_i].
		$$
		Thus, we have
		$$
		(M_{i+1} - m_i+1)+(n - k - |V( Q_i)|) \leq n - k,
		$$
		which implies 
		$$
		|V(Q_i)| \geq M_{i+1} - m_i+1 \geq \frac{d_{P}(v_i) + d_{P}(v_{i+1})+ \varepsilon_i + \varepsilon_{i+1}}{2}.
		$$
		This completes the case $r_i=0$.

		Suppose next that $r_i \geq 1$.
		Let $w$ and $w'$ be the neighbors of $u_1$ in $Q_i$
		such that $\overrightarrow{C_0}[v_i^+,w]$ and $\overrightarrow{C_0}[w', v_{i+1}^-]$
		are as short as possible, respectively.
		Note that possibly $w = w'$.
		By the definition,
		all neighbors of $u_1$ in $Q_i$ are contained in $\overrightarrow{C_0}[w, w']$.
		Since 
		$C_0$ is a longest cycle in $G$,
		no two consecutive vertices of $C_0$ can both be adjacent to $u_1$. 
		Thus, we have
		$|V(\overrightarrow{C_0}[w,w'])| \geq 2r_i -1$.
		
		Recall that $M_i$ denote the largest index such that $u_{M_i} \in N_{P}(v_i)$,
		which implies $M_i \geq d_{P'}(v_i) +1$.
		Since $G$ has circumference $n-k,$ the following cycle has order at most $n-k$.
		$$
		v_iu_{M_i}\cup P[u_{M_i},u_1]\cup u_1w \cup \overrightarrow{C_0}[w,v_i]
		$$
		Hence,
		\begin{eqnarray*}
			M_{i} + \left(n - k - |V(\overrightarrow{C_0}[v_i^+,w^-])|\right) \leq n - k,
			\quad 
			\text{ implying } 
			\quad 
			|V(\overrightarrow{C_0}[v_i^+,w^-])| \geq M_i \geq d_{P'}(v_i) +1.
		\end{eqnarray*}
		Similarly,
		we obtain 
		$$
		|V(\overrightarrow{C_0}[{w'}^+,v_{i+1}^-])| \geq d_{P'}(v_{i+1}) + 1.
		$$
		%
		Therefore, \begin{eqnarray*}\label{eq3}
			|V(Q_i)| &=&
			|V(\overrightarrow{C_0}[v_i^+,w^-])| +
			|V(\overrightarrow{C_0}[w,w'])| + 
			|V(\overrightarrow{C_0}[{w'}^+,v_{+1}^-])| 
			\\
			&\ge& 
			\left(d_{P'}(v_i) + 1\right) + \left(2r_i -1\right) + \left(d_{P'}(v_{i+1}) + 1\right)
			\\
			&=& 
			d_{P'}(v_i) + d_{P'}(v_{i+1}) + 1 + 2r_i.
		\end{eqnarray*}
		Since $\frac{\varepsilon_i +\varepsilon_{i+1}}{2} \leq 1$,
		this implies the desired inequality.
	\end{proof}
	\medskip
	
	Recall that 
	$\sum_{i=1}^s \varepsilon_i \leq s$ and 
	$\Delta(G) = d_G(u_1) = 1+ \sum_{i=1}^s (r_i + \varepsilon_i)$.
	By Claim \ref{claim:Qi},
	we can derive the following estimate.
	\begin{align*}
		n - k &= |V(C_0)| = s + \sum_{i=1}^{s} \left| V\left( Q_i \right) \right| \\
		&\geq \sum_{i=1}^s \varepsilon_i + \sum_{i=1}^{s} \left( \frac{d_{P'}(v_i)+ d_{P'}(v_{i+1})+ \varepsilon_i + \varepsilon_{i+1}}{2} +2r_i\right)
		\\
		&\geq \sum_{i=1}^{s} d_{P'}(v_i) + 2 \sum_{i=1}^s \varepsilon_i +2r_i
		\\
		&=\sum_{i=1}^{s} d_{P'}(v_i) +2\Delta(G) -2.
	\end{align*}
	
	Since $|N_G(u_j)\cap V(C_0)|\geq (k+2)-2=k$ for $2\leq j\leq k-1$ and $|N_G(u_k)\cap V(C_0)|\geq (k+2)-1=k+1$, we have that 
	$$
	\sum_{i=1}^{s} d_{P'}(v_i)=|E_G(P',C_0)|=\sum_{j=2}^k |N_G(u_j)\cap V(C_0)|\geq (k-2)k+(k+1)=k^2-k+1.
	$$
	Combining these, we conclude that
	\begin{equation*}\label{eq5}
		n - k -2\Delta(G)+ 2\geq k^2-k+1 .
	\end{equation*}
	Then, we have 
	$$
	\Delta(G) \leq  \frac{n - k^2 +1}{2} .
	$$
	This completes the proof. \hfill $\Box$
	\bigskip
	
	Note that for any $G \in \varPi(n; k)$, the deletion of an arbitrary set of $k$ vertices results in a traceable—and hence connected—subgraph. In particular, this implies that $G$ is $(k+1)$-connected.  Arguing as in Property~\ref{lem1}, we conclude that $G$ contains an induced path with order at least $k+1$. 
	\begin{property}\label{lem3}
		Let $n$ and $k$ be a positive integer with $n\geq k\geq 2$. If $G\in \varPi(n;k)$, then for every vertex $u$ of $G$, $G$ has an induced path of order at least $k+1$ with an endpoint $u$.
	\end{property}
	We now turn to the proof of Theorem~\ref{thm4}. Since the argument closely parallels that of Theorem~\ref{thm3}, we omit some routine details here.
	\bigskip
	
	\noindent{\it Proof of Theorem \ref{thm4}.}  Let $G$ be a graph in $\varPi(n; k)$, and let $x_1$ denote a vertex of maximum degree in $G$. By the definition of $\varPi(n; k)$, the detour order of $G$ is $n-k$. Since $G$ is $(k + 1)$-connected, it follows that the minimum degree is at least  $k + 1$.
	
	By property \ref{lem3}, there exists an induce path $P$ of order $k$ with vertex sequence $x_1, x_2, \ldots, x_k$, where $x_i \sim x_{i+1}$ for all $1 \le i < k$.  Let $P_0$ be a Hamilton path in the induced subgraph $G - V(P)$. Fix an orientation of $P_0$, and denote the resulting directed path by $\overrightarrow{P_0}$.  For a vertex $y \in V(P_0)$, we write $y^+$ for the successor of $y$ and $y^-$ for the predecessor of $y$ along $\overrightarrow{P_0}$. Set $P'=P-x_1$. Denote
	$$
	N_G(P') \cap V(P_0) = \{ y_1, y_2, \ldots, y_s \},
	$$
	where the indices are ordered increasingly along the orientation. Since $G$ is $(k + 1)$-connected and $P'$ is disjoint from $P_0$, it follows that $s \geq k$.  For convenience, let $y_0$ and $y_{s+1}$ denote the two endpoints of the path $P_0$  where $y_0$ is the source of $\overrightarrow{P_0}$. Since the detour order of $G$ is $n - k$, we have  $y_0 \neq y_1$ and $y_{s+1} \neq y_s$.  Otherwise, a longer path could be constructed. For each $y_i \in N_G(P') \cap V(P_0)$, define
	$$
	N_{P}(y_i) = N_G(y_i) \cap V(P), \quad d_{P'}(y_i) = |N_{P}(y_i) \setminus \{x_1\}|,
	\quad \text{and}  \quad \varepsilon_i = 
	\begin{cases} 
		1 & \text{if $x_1$ is a neighbor of $y_i$,} \\
		0 & \text{otherwise}.
	\end{cases} 
	$$
	
	For each $i\in \{1,2,\ldots,s-1 \}, $  we denote $W_i=\overrightarrow{P_0}[y^+_i,y^-_{i+1}]$ and let $r_i = |V(W_i) \cap N_G(x_1)|$. By applying the same argument as in Claim \ref{claim:Qi} of Theorem \ref{thm3}, we obtain the following claim.
	
	\medskip
	\begin{claim}\label{claim2}
		For each $i \in \{1, \ldots, s-1\}$, we have
		\begin{equation*}\label{eq1}
			|V( W_i )| \geq \frac{d_{P'}(y_i) + d_{P'}(y_{i+1})+ \varepsilon_i +\varepsilon_{i+1}}{2} + 2r_i.
		\end{equation*}
	\end{claim}
	
	Let $W_0=\overrightarrow{P_0}[y_0,y_1^{-}]$, $W_s=\overrightarrow{P_0}[y_s^{+}, y_{s+1}]$ and $r_i=|V(W_i)\cap N_G(x_1)|$ for $i\in \{0,s\}$.
	In order to estimate the number of edges in the path $P_0$, we also need the following claim. 
	
	\medskip
	\begin{claim}\label{claim3}
		We have 
		$$
		|V(W_0)| \ge \frac{d_{P'}(y_1) + k+ \varepsilon_1}{2}+2r_0, 
		\qquad
		|V(W_s)| \ge \frac{d_{P'}(y_s) + k+ \varepsilon_{s}}{2}+2r_{s}.
		$$
	\end{claim}
	
	\begin{proof}
		Let $m_1$ be the smallest index with $x_{m_1} \in N_{P}(y_1)$, and $M_1$ the largest index with $x_{M_1} \in N_{P}(y_1)$. Consider the two paths constructed from $P_0$ in the following manner. Since $P_0$ is a longest path, the orders of these two paths cannot exceed $|V(P_0)|$.
		
		$$
		P[x_k, x_{m_1}] \cup x_{m_1} y_1 \cup  P_0 [y_1, y_{s+1}]
		\quad\text{and}\quad
		P[x_1, x_{M_1}] \cup x_{M_1} y_1 \cup   P_0 [y_1,y_{s+1}].
		$$
		
		Suppose that $r_0 = 0$. In this case, we have
		$$
		|V(W_0)| \ge |V(P[x_k, x_{m_1}])|\geq  k-m_1 +1\quad\text{and}\quad |V(W_0)| \ge |V(P[x_1, x_{M_1}])|\geq M_1.
		$$
		Hence 
		\begin{align*}
			|V(W_0)| 
			\ge \frac{k-m_1+1+M_1}{2}  
			= \frac{k + M_1 - m_1 + 1}{2} 
			\ge \frac{k + d_{P'}(y_1)+\varepsilon_1}{2}
			= \frac{k + d_{P'}(y_1)+\varepsilon_1}{2}+2r_0.
		\end{align*}
		
		Now suppose that $r_0\ge 1$, and let $w$ and $w'$ be the neighbors of $x_1$ in $W_0=\overrightarrow{P_0}[y_0, y_1^{-}]$ such that $\overrightarrow{P_0}[y_0,w]$ and $\overrightarrow{P_0}[w', y_1^-]$ are as short as possible, respectively. We can deduce that 
		\begin{align*}
			|V(W_0)| &= |V(\overrightarrow{P_0}[y_0, w^{-}])|+|V(\overrightarrow{P_0}[w, w'])|+|V(\overrightarrow{P_0}[w'^+, y_1^-])|\\
			&\geq  k+2r_0-1+d_{p'}(y_1)+1\\
			&\geq \frac{k + d_{P'}(y_1)+\varepsilon_1}{2}+2r_0.
		\end{align*}

		The same argument applied to $W_s$ yields
		$$
		|V(W_s)| \ge \frac{k + d_{P'}(y_s)+\varepsilon_s}{2}+2r_s.
		$$
	\end{proof}
	
	By Claims~\ref{claim2} and~\ref{claim3}, since $s\geq \sum_{i=1}^{s} \varepsilon_i$ and  $\Delta(G)=1+\sum_{i=1}^{s} \varepsilon_i+\sum_{i=0}^{s}r_i$, we have
	\begin{align*}
		n - k &= |V(P_0)| \\
		&= s+	|V( W_0) | + \sum_{i=1}^{s-1} \left| V\left(W_i  \right) \right| +	|V(W_s)| \\
		&\geq s+ \frac{k + d_{P'}(y_1)+\varepsilon_1}{2}+2r_0+\sum_{i=1}^{s-1} \left( \frac{d_{P'}(y_i) + d_{P'}(y_{i+1})+ \varepsilon_i + \varepsilon_{i+1}}{2}
		+ 2r_i\right) \\
		&\quad +\frac{k + d_{P'}(y_s)+\varepsilon_s}{2}+2r_s\\
		&= s+k+\sum_{i=1}^{s} d_{P'}(y_i) +\sum_{i=1}^{s} \varepsilon_i+2\sum_{i=0}^{s}r_i\\
		&\ge k+|E_G(P', P_0)| + 2\Delta(G)-2.
	\end{align*}
	and hence
	$$ n - 2k - 2\Delta(G)+2\geq |E_G(P', P_0)| . $$
	
	On the other hand, since $P'$ is an induced path, we have
	\begin{align*}
		|E_G(P', P_0)| \geq k+(k-2)(k-1).
	\end{align*}
	Combining these, we conclude that
	$$
	n - 2k - 2\Delta(G)+2\geq k+(k-2)(k-1) ,
	$$
	which forces
	$$
	\frac{n-k^2}{2}\geq \Delta(G).
	$$
	This completes the proof.   \qed
	
	\section*{Acknowledgement} Masaki Kashima was supported by JSPS KAKENHI, Grant number 25KJ2077. Kenta Ozeki was supported by JSPS KAKENHI, Grant Numbers 22K19773 and 23K03195. Leilei Zhang's work was supported by JSPS KAKENHI Grant Number 25KF0036, the NSF of Hubei Province Grant Number 2025AFB309,  the China Postdoctoral Science Foundation  Grant Number 2025M773113, the Fundamental Research Funds for the Central Universities, Central China Normal University Grant Number CCNU24XJ026.

	\section*{Declarations}
	
	\begin{itemize}
		\item \textbf{Conflict of interest}\quad The authors declare that they have no known competing financial interests or personal
		relationships that could have appeared to influence the work reported in this paper.
		\item \textbf{Data availibility statement}\quad This manuscript has no associated data.
	\end{itemize}

\end{document}